\documentclass[12pt]{article}

\usepackage{amsmath,euscript}
\usepackage{amssymb}
\usepackage{amsthm}
\usepackage{enumerate}
\usepackage{amsfonts}
\usepackage{comment}
\usepackage{colortbl}
\usepackage{amssymb,amsmath,array}

\usepackage[enableskew]{youngtab}
\usepackage{latexsym}

\usepackage{amscd}

\usepackage[latin1]{inputenc}
\usepackage{pstricks,pst-node,pst-text,pst-3d}
\usepackage{amsmath}
\usepackage{pgf}

\input{xy}
\xyoption{all}

\newcommand{\bbox}{\rule[-0.1em]{0.85em}{0.85em}}
\newcommand{\wbox}{\mbox{}}

\newcommand{\fin}{\hfill$\square$}

\newcommand{\gc}{ [ \hspace{-0.65mm} [}
\newcommand{\dc}{]  \hspace{-0.65mm} ]}

\newcommand{\mmpc}{O_{q}\left( \mathcal{M}_{m,p}(\mathbb{C}) \right)}

\def\mmp{M_{m,p}(\mr)}

\newcommand{\oc}{\mathcal{O}}
\newcommand{\mc}{\mathcal{M}}

\newcommand{\ia}{i,\alpha}
\newcommand{\hc}{\mathcal{H}}
\newcommand{\ch}{\mathcal{H}}
\def\tnn{{\rm tnn}}

\newtheorem{theorem}{Theorem}[section]
\newtheorem{proposition}[theorem]{Proposition}

\theoremstyle{definition}
\newtheorem{definition}[theorem]{Definition}
\newtheorem{notation}[theorem]{Notation}

\newtheorem{example}[theorem]{Example}
\newtheorem{exercise}[theorem]{Exercise}

\newcommand{\bbR}{\mathbb{R}} 
\newcommand{\bbC}{\mathbb{C}} 
\newcommand{\bbK}{\mathbb{C}} 

\newcommand{\hspec}{{\mathcal H}-{\rm Spec}}
\newcommand{\Mmpc}{\mc_{m,p}(\bbC)}
\newcommand{\OMmpc}{\oc(\Mmpc)}
\newcommand{\Mtwoc}{\mc_{2}(\bbC)}
\newcommand{\OMtwoc}{\oc(\Mtwoc)}

\newcommand{\oqmtwo}{\oc_q(\Mtwoc)}
\newcommand{\oqmmpc}{\oc_q(\Mmpc)}

\def\mmp{\mc_{m,p}}
\def\mtnnmp{\mmp^{\rm tnn}}

\newcommand{\lc}{\mathcal{L}}
\newcommand{\wop}{w_\circ^p}
\newcommand{\wom}{w_\circ^m}
\newcommand{\woN}{w_\circ^N}

\newcommand{\onem}{\gc 1,m \dc}
\newcommand{\onep}{\gc 1,p \dc}

\input xy \xyoption{matrix} \xyoption{arrow}
\def\edge{\ar@{-}}
\def\plc{*{}+0}

\def\bigblacksquare{\vrule height 8pt depth 2pt width 10pt}
\def\mdblk{\save+<0ex,0ex>\drop{\bigblacksquare}\restore}

\input pstricks

\title{From totally nonnegative matrices to quantum matrices and back, via 
Poisson geometry
}

\author{S. Launois\thanks{\,The research of the first named author was
supported by a Marie Curie European Reintegration Grant within the
$7^{\mbox{th}}$ European Community Framework Programme.}~~and T.H. Lenagan}

\date{}

\begin{document}
\maketitle

\abstract{\footnotesize In this survey article, we describe recent work that connects three separate objects of interest: totally nonnegative matrices; quantum matrices; and matrix Poisson varieties.}\\



\section*{Introduction} 

In recent publications, the same combinatorial
description has arisen for three separate objects of interest: $\hc$-prime ideals in quantum matrices, $\hc$-orbits of symplectic leaves in matrix Poisson varieties and totally nonnegative cells in the space of totally nonnegative matrices.  

Many quantum algebras have a natural action by a torus and a key ingredient in the study of the structure of these algebras is an understanding of the torus-invariant objects. For example, the Stratification Theory of Goodearl and Letzter shows that, in the generic case, a complete understanding of the prime
spectrum of quantum matrices would start by classifying the (finitely many) torus-invariant prime ideals. In \cite{Cau2}
Cauchon succeeded in counting the number of torus-invariant prime ideals in
quantum matrices. His method involved a bijection between certain diagrams,
now known as Cauchon diagrams, and the torus-invariant primes. Considerable progress in the understanding of quantum matrices has been made since that time by using Cauchon diagrams. 

The semiclassical limit of quantum matrices is the classical coordinate ring
of the variety of matrices endowed with a Poisson bracket that encodes the
nature of the quantum deformation which leads to quantum matrices. As a
result, the variety of matrices is endowed with a Poisson structure. A natural
torus action leads to a stratification of the variety via torus-orbits of
symplectic leaves. In \cite{BGY}, Brown, Goodearl and Yakimov showed that
there are finitely many such torus-orbits of symplectic leaves. Each torus
orbit is defined by certain rank conditions on submatrices. The classification
is given in terms of certain permutations from the relevant symmetric group
with restrictions arising from the Bruhat order. 

The nonnegative part of the space of real matrices consists of those matrices
whose minors are all nonnegative. One can specify a cell decomposition of the
set of totally nonnegative matrices by specifying exactly which minors are to
be zero/non-zero. In \cite{Pos}, Postnikov classified the nonempty cells by
means of a bijection with certain diagrams, known as Le-diagrams. The work of Postnikov was then developed by Talaska \cite{Tal}, Williams \cite{Wil}, etc., and led to the definition of {\it positroid varieties} that have been recently studied by Oh \cite{Oh} and Knutson, Lam and Speyer \cite{KLS}.

The interesting observation from the point of view of this work is that in
each of the above three sets of results the combinatorial objects that arise
turn out to be the same! The definitions of Cauchon diagrams and Le-diagrams
are the same, and the restricted permutations arising in the
Brown-Goodearl-Yakimov study can be seen to lead to Cauchon/Le diagrams via
the notion of pipe dreams.

Once one is aware of these connections, this suggests that there should be a
connection between torus-invariant prime ideals, torus-orbits of symplectic
leaves and totally nonnegative cells. This connection has been investigated in
recent papers by Goodearl and the present authors, \cite{GLL, GLL2}. In
particular, we have shown that the Restoration Algorithm, developed by the
first author for use in quantum matrices, can also be used in the other two
settings to answer questions concerning the torus-orbits of symplectic leaves
and totally nonnegative cells. The detailed proofs of the results that were
obtained in \cite{GLL, GLL2} are very technical, and our aim in this survey,
is to describe the results informally and to compute some examples to
illuminate our results. 

\section{Totally nonnegative matrices} 

A matrix is {\em totally positive} (TP for short) if each of its minors is positive
and is {\em totally nonnegative} (TNN for short) if 
each of its minors is nonnegative. 

An excellent survey of totally positive and totally nonnegative
matrices can be found in \cite{FZ}. In this survey, the authors draw 
attention to appearance of TP and TNN matrices in many areas of mathematics, 
including: oscillations in mechanical systems, stochastic processes and 
approximation theory, P\'olya frequency sequences, representation theory, 
planar networks, ... . 
A good source of examples, especially illustrating the important link with 
planar networks (discussed below) is 
\cite{Sk}.

\subsection{Checking total positivity and total nonnegativity}

Let us start with an example.

\begin{example}\label{example-4x4}\rm  cf. \cite{Sk}. Is the matrix 
\[A:= 
\left(\begin{array}{cccc}
5&6&3&0\\
4&7&4&0\\
1&4&4&2\\
0&1&2&3
\end{array}\right)
\]
totally nonnegative? 

In order to check this by calculating all minors, we would have to
calculate 
\[
{4\choose 1}^2 + {4\choose 2}^2+{4\choose 3}^2+{4\choose 4}^2
= 16+36+16+1=69\]
minors. In general, the number of minors of an $n\times n$ matrix is 
\[
\sum_{k=1}^n\,{n\choose k}^2 
\quad=\quad {2n\choose n} -1
\quad\approx \quad \frac{4^n}{\sqrt{\pi n}}   
\]
by using Stirling's approximation
\[
n!\approx\sqrt{2\pi n}\frac{n^n}{e^n}.
\]
This suggests that we do not want to calculate all of the minors to
check for total nonnegativity.  \fin\\
\end{example} 

Luckily, for total positivity, we can get away with much
less. The simplest example is the $2\times 2$ case. 

%%%%%%%%%%%%%%%%%%%%%%%%%%%%%%%%%%%%%%%%%%%%%%%%%%%%%%%%%%%

\begin{example}\label{example-2x2}\rm

The matrix 
\[
\left(\begin{array}{cc}
a&b\\
c&d
\end{array}\right)
\]
has {\em five} minors: $a,b,c,d,\Delta = ad-bc$.

If $a,b,c,\Delta = ad-bc >0$ then 
\[
d=\frac{\Delta+bc}{a} >0\]
so it is sufficient to check {\em four} minors. \fin\\
\end{example}

The optimal result is due to Gasca and Pe\~na, \cite[Theorem 4.1]{GaPe}: 
for an $n\times n$ matrix, it is only necessary to
check $n^2$ specified minors.

\begin{definition}\rm 
A minor is said to be an {\em initial minor} if it is formed of consecutive rows and columns, one of which being the first row or the first column. 
\end{definition}

For example, a $2\times 2$ matrix has 4 initial minors: $a$, $b$, $c$ and $\Delta$. More generally, an initial minor is specified by its bottom right entry; so an $n\times n$ matrix has $n^2$ initial minors. 

\begin{theorem} (Gasca and Pe\~na) 
The $n\times n$ matrix $A$ is totally positive  if and only if each 
of its initial minors is positive.
\end{theorem} 

There is no such family to check whether a matrix is TNN. 
However Gasca and Pe\~na do give an efficient algorithm to check TNN,
see the comment after \cite[Theorem 5.4]{GaPe}.

\subsection{Planar networks} 

We refer the reader to \cite{Sk} for the definition of a planar network. 
Consider a directed planar graph with no directed cycles, 
$m$ sources, $s_i$ and $n$ sinks $t_j$. See Figure~\ref{fig:planarnet} 
(taken from \cite{Sk}) 
for an example.

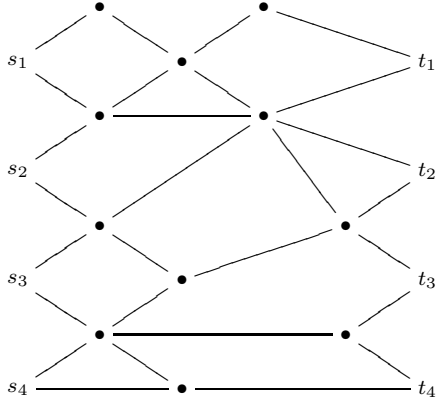
\begin{figure}
$$
\xymatrixrowsep{4ex}\xymatrixcolsep{6ex}\def\objectstyle{\scriptstyle}
\xymatrix@!0
{
 &\bullet\edge[dl]\edge[dr]&&\bullet\edge[dl]\edge[drr]&&\\
 %%%
 s_1\edge[dr] &&\bullet\edge[dl]\edge[dr]&&&t_1\edge[dll]\\
 %%%
 &\bullet\edge[dl]\edge[rr]&&\bullet\edge[ddll]\edge[ddr]\edge[drr]&&\\
 %%%
 s_2\edge[dr]&&&&&t_2\edge[dl]\\
 %%%
 &\bullet\edge[dl]\edge[dr]&&&\bullet\edge[dll]\edge[dr]&\\
 %%%
 s_3\edge[dr]&&\bullet\edge[dl]&&&t_3\edge[dl]\\
 %%%
 &\bullet\edge[dl]\edge[dr]\edge[rrr]&&&\bullet\edge[dr]&\\
 %%%
 s_4\edge[rr]&&\bullet\edge[rrr]&&&t_4
}
$$
\caption{\label{fig:planarnet}An example of a 
planar network (all edges are directed left to right)}
\end{figure}

Set {$M=\left(m_{ij}\right)$ 
 where $m_{ij}$ is the number of paths from source $s_i$ 
 to sink $t_j$. The matrix $M$ is called the {\em path matrix} of this planar network.

\begin{notation}\rm
The minor formed by using rows from a set $I$ and columns 
from a set $J$ is denoted by $[I\mid J]$.
\end{notation}

Planar networks give an easy way to construct TNN matrices.

\begin{theorem}(Lindstr\"om's Lemma, \cite{Lind}) 
\label{lindstrom}
The path matrix of any planar network is totally nonnegative. 
In fact, the minor  $[I\mid J]$ is equal to the number of families of 
non-intersecting paths from sources indexed by $I$ and sinks indexed 
by $J$. 
\end{theorem} 

If we allow weights on paths then even more is true.

\begin{theorem} (Brenti, \cite{Bre})
Every totally nonnegative matrix is the weighted path matrix of some 
planar network. 
\end{theorem}

\begin{example} 
This example is taken from \cite{Sk}. The path matrix 
\[
 M= \left(\begin{array}{cccc}
5&6&3&0\\
4&7&4&0\\
1&4&4&2\\
0&1&2&3
\end{array}\right)
\]
of the planar network in Figure \ref{fig:planarnet} is the matrix of Example~\ref{example-4x4}. Thus, $M$ is totally
nonnegative, by Lindstr\"om's Lemma.\fin 
\end{example} 

%%%

\subsection{Cell decomposition}

Our main concern in this section is to consider the possible patterns of zeros
that can occur as the values of the minors of a totally nonnegative matrix.
The following example shows that one cannot choose a subset of minors
arbitrarily and hope to find a totally nonnegative matrix for which the chosen 
subset is precisely the subset of minors with value zero.

\begin{example}\label{example:celld}\rm  There is no $2\times 2$ totally nonnegative matrix 
$\left( \begin{array}{cc} a&b\\c&d\end{array} \right)$ with $d=0$, but the
other four minors nonzero.
For, suppose that $\left( \begin{array}{cc} a&b\\c&d\end{array} \right)$ is
TNN and $d=0$. 
Then $a,b,c\geq 0$ and also $ad-bc\geq 0$. 
Thus, $-bc \geq 0$
and hence $bc =0$ so that $b=0$ or $c=0$. \fin\\
\end{example}

Let $\mtnnmp$ be the set of totally nonnegative $m\times p$ real matrices.
Let $Z$ be a subset of minors. The {\em cell} $S_Z^o$  is 
the set of matrices 
in $\mtnnmp$ for which the minors in $Z$ are zero (and those not in $Z$ are 
nonzero). Some cells may be empty. The space $\mtnnmp$ is partitioned by 
the nonempty cells. \\

%%%%%%%%%%%%%%%%%%%%%%%%%%%%%%%%%%%%%%%%%%%%%%%%%%%%%%%%%%%
%%%%%%%%%%%%%%%%%%%%%%%%%%%%%%%%%%%%%%%%%%%%%%%%%%%%%%%%%%%

\begin{exercise} \label{exercise-14}\rm
Show that there are $14$ nonempty cells in $\mathcal{M}_2^{\tnn}$. \fin 
\end{exercise} 

The question is then to describe the patterns of minors that represent nonempty cells
in the space of totally nonnegative matrices.  In \cite{Pos}, Postnikov defines {\em Le-diagrams} 
to solve this problem. An $m\times p$ array with entries either $0$ or $1$ is said to be a 
{\em 
Le-diagram} if it satisfies the following rule: if there is a $0$ in a given
square then either each square to the left is also filled with $0$ or each
square above is also filled with $0$. 

Here are an 
example and a non-example of a Le-diagram on a $5\times 5$ array.

\[ 
\young(11010,00010,11110,00010,11110)
\qquad\qquad
\young(11010,00101,11101,00111,11111)
\]

% 

%%%%%%%%%%%%%%%%%%%%%%%%%%%%%%%%%%%%%%%%%%%%%%%%%%%%%%%%%%%

%%%%%%%%%%%%%%%%%%%%%%%%%%%%%%%%%%%%%%%%%%%%%%%%%%%%%%%%%%%

% 
\begin{theorem} (Postnikov) 
\label{thm:postnikov}
There is a bijection between Le-diagrams on an 
$m\times p$ array and nonempty cells $S^{\circ}_Z$ in
$\mtnnmp$. 
\end{theorem} 

In fact, Postnikov proves this theorem for the totally nonnegative grassmannian, and we are interpreting the result on the big cell, which is 
the space of totally nonnegative matrices. 

In view of Exercise~\ref{exercise-14}, there should be $14$ $2\times 2$
Le-diagrams.

\begin{exercise} \label{exercise-16-to-14}\rm 
The $16$ possible fillings of a $2\times 2$ array with either $0$ or $1$ are
shown in Figure 2. Identify the two non-Le-diagrams.\fin

\begin{figure}
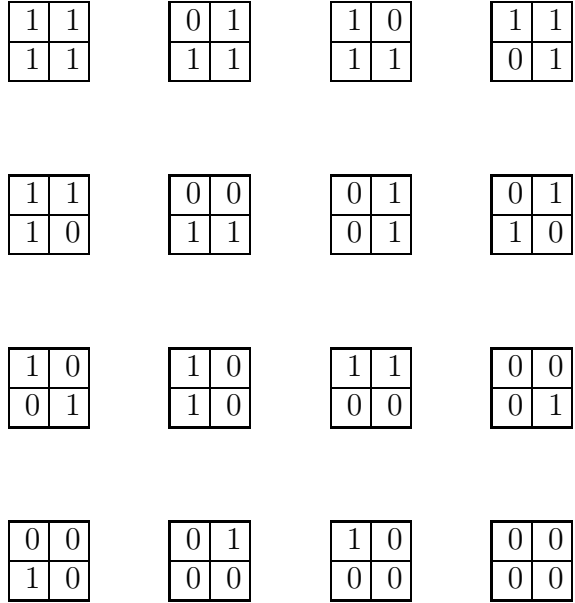

%\label{fig:planarnet}
\begin{center}
$\begin{array}{lcclcclccl}
\begin{tabular}{| b{1mm} | b{1mm} | }
    \hline
     1 & 1\\
    \hline
    1   & 1 \\
    \hline
  \end{tabular} & \hspace{3ex} & &
\begin{tabular}{| b{1mm} | b{1mm} | }
    \hline
     0 & 1\\
    \hline
    1   & 1 \\
    \hline
  \end{tabular} & \hspace{3ex} & &
\begin{tabular}{| b{1mm} | b{1mm} | }
    \hline
     1 & 0\\
    \hline
    1   & 1 \\
    \hline
  \end{tabular}& \hspace{3ex} &  & \begin{tabular}{| b{1mm} | b{1mm} | }
    \hline
     1 & 1\\
    \hline
    0   & 1 \\
    \hline
  \end{tabular} \\

\end{array}$
\end{center}

\vspace{2ex}

\begin{center}
$\begin{array}{lcclcclccl}
\begin{tabular}{| b{1mm} | b{1mm} | }
    \hline
     1 & 1\\
    \hline
    1   & 0 \\
    \hline
  \end{tabular} & \hspace{3ex} & &
\begin{tabular}{| b{1mm} | b{1mm} | }
    \hline
     0 & 0\\
    \hline
    1   & 1 \\
    \hline
  \end{tabular} & \hspace{3ex} & &
\begin{tabular}{| b{1mm} | b{1mm} | }
    \hline
     0 & 1\\
    \hline
    0   & 1 \\
    \hline
  \end{tabular}& \hspace{3ex} &  & \begin{tabular}{| b{1mm} | b{1mm} | }
    \hline
     0 & 1\\
    \hline
    1  & 0 \\
    \hline
  \end{tabular} \\

\end{array}$
\end{center}

\vspace{2ex}

\begin{center}
$\begin{array}{lcclcclccl}
\begin{tabular}{| b{1mm} | b{1mm} | }
    \hline
     1 & 0\\
    \hline
    0  & 1 \\
    \hline
  \end{tabular} & \hspace{3ex} & &
\begin{tabular}{| b{1mm} | b{1mm} | }
    \hline
     1 & 0\\
    \hline
    1   & 0 \\
    \hline
  \end{tabular} & \hspace{3ex} & &
\begin{tabular}{| b{1mm} | b{1mm} | }
    \hline
     1 & 1\\
    \hline
    0  & 0 \\
    \hline
  \end{tabular}& \hspace{3ex} &  & \begin{tabular}{| b{1mm} | b{1mm} | }
    \hline
     0 & 0\\
    \hline
    0   & 1 \\
    \hline
  \end{tabular} \\

\end{array}$
\end{center}

\vspace{2ex}

\begin{center}
$\begin{array}{lcclcclccl}
\begin{tabular}{| b{1mm} | b{1mm} | }
    \hline
     0 & 0\\
    \hline
    1   & 0 \\
    \hline
  \end{tabular} & \hspace{3ex} & &
\begin{tabular}{| b{1mm} | b{1mm} | }
    \hline
     0 & 1\\
    \hline
    0  & 0 \\
    \hline
  \end{tabular} & \hspace{3ex} & &
\begin{tabular}{| b{1mm} | b{1mm} | }
    \hline
     1 & 0\\
    \hline
    0  & 0 \\
    \hline
  \end{tabular}& \hspace{3ex} &  & \begin{tabular}{| b{1mm} | b{1mm} | }
    \hline
     0 & 0\\
    \hline
    0   & 0 \\
    \hline
  \end{tabular} \\

\end{array}$
\end{center}\caption{$2\times 2$ Le-diagrams} 
\end{figure}
\end{exercise} 

%%%%%%%%%%%%%%%%%%%%%%%%%%%%%%%%%%%%%%%%%%%%%%%%%%%%%%%%%%%

In \cite{Pos}, Postnikov describes an algorithm that starts with a Le-diagram
and produces a planar network from which one generates a totally nonnegative
matrix which defines a nonempty cell in the space of totally nonnegative
matrices. The procedure to produce the planar network is as follows. In each
1 box of the Le-diagram, place a black dot. From each black dot draw a
hook which goes to the right end of the diagram and the bottom of the diagram.
Label the right ends of the horizontal part of the hooks as the sources of a
planar network, numbered from top to bottom, and label the bottom ends of the
vertical part of the hooks as the sinks, numbered from left to right. Then
consider the resulting graph to be directed by allowing movement from right to
left along horizontal lines and top to bottom along vertical lines. By Lindstr\"om's Lemma (see Theorem \ref{lindstrom}) the path matrix of this planar network is a totally nonnegative matrix, and so the pattern of its zero minors produces a nonempty cell in the space of totally nonnegative matrices. The above procedure that associates to any Le-diagram a nonempty cell provides a  bijection between the set of $m \times p$ Le-diagrams and nonempty cells in the space of totally nonnegative $m \times p$ matrices (see Theorem \ref{thm:postnikov}).  
We illustrate Postnikov's procedure with an example.

\begin{example}\label{example-postnikov-algorithm} 
The Le-diagram  
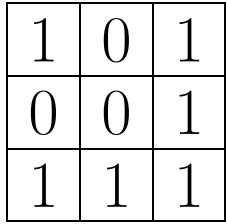
\begin{center}
{\Huge 
\young(101,001,111)
}
\end{center}
produces the following planar network

$$
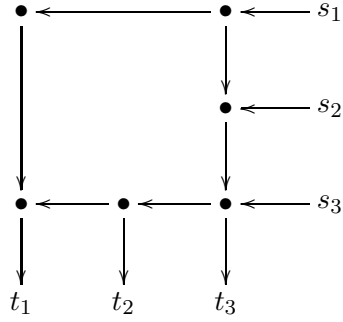
\xymatrix{
 \bullet \ar@{->}[dd]& & \bullet \ar@{->}[ll]\ar@{->}[d] & s_1\ar@{->}[l]   \\
& & \bullet \ar@{->}[d] & s_2\ar@{->}[l] \\
\bullet \ar@{->}[d]& \bullet \ar@{->}[l]\ar@{->}[d]& \bullet\ar@{->}[l]\ar@{->}[d] & s_3\ar@{->}[l]\\
t_1 & t_2 & t_3 & \\
}$$
where $s_1$, $s_2$ and $s_3$ are the sources and $t_1$, $t_2$ and $t_3$ are the sinks. The path matrix of this planar network is
\[
\left(\begin{array}{ccc}
2&1&1\\
1&1&1\\
1&1&1
\end{array}\right) .
\]
The minors that vanish on this matrix are:
$$[1,2|2,3],[1,3|2,3],[2,3|2,3],[2,3|1,3],[2,3|1,2],[1,2,3|1,2,3].$$
The cell associated to this family of minors is nonempty and this is the 
nonempty cell associated to the Le-diagram above. \fin\\
\end{example}  

In fact, by allowing suitable weights on the edges of the above planar
network, one can obtain all of the matrices in this cell as weighted path
matrices of the planar network. 

%%%%%%%%%%%%%%%%%%%%%%%%%%%%%%%%%%%%%%%%%%%%%%%%%%%%%%%%%%%

\section{Quantum matrices}

We denote by $R:=\mmpc$ the standard quantisation of the ring of
regular functions on $m \times p$ matrices with entries in $\bbK$; 
the algebra $R$ is
the $\bbK$-algebra generated by the $m \times p $ indeterminates
$X_{\ia}$, for 
$1 \leq i \leq m$ and $ 1 \leq \alpha \leq p$, subject to the
following relations:
 \[
\begin{array}{ll}
X_{i, \beta}X_{i, \alpha}=q^{-1} X_{i, \alpha}X_{i ,\beta},
& (\alpha < \beta); \\
X_{j, \alpha}X_{i, \alpha}=q^{-1}X_{i, \alpha}X_{j, \alpha},
& (i<j); \\
X_{j,\beta}X_{i, \alpha}=X_{i, \alpha}X_{j,\beta},
& (i <j,\;  \alpha > \beta); \\
X_{j,\beta}X_{i, \alpha}=X_{i, \alpha} X_{j,\beta}-(q-q^{-1})X_{i,\beta}X_{j,\alpha},
& (i<j,\;  \alpha <\beta). 
\end{array}
\]

It is well known that $R$ can be presented as an iterated Ore extension over
$\bbK$, with the generators $X_{\ia}$ adjoined in lexicographic order.
Thus, the ring $R$ is a noetherian domain; its skew-field of
fractions is denoted by $F$ or ${\rm F}(R)$. 
In the case that  $q$ is not a root of unity, it follows from
\cite[Theorem 3.2]{GoLet1} that all prime ideals of $R$ are completely
prime. In this survey, we will assume that $q$ is not a root of
unity.

Let $K$ be a $\bbC$-algebra and
$M=(x_{i,\alpha})\in \mc_{m,p}(K)$. If $I \subseteq \onem$ and $\Lambda
\subseteq \onep$ with $|I|=|\Lambda |=t \geq 1$, then we denote by $[I |
\Lambda ]_q(M)$ the corresponding \emph{quantum minor} of $M$. This is the element of
$K$ defined by: 
$$[I | \Lambda ]_q(M)= [i_1, \dots, i_k|\alpha_1 , \dots , \alpha_k]_q :=\sum_{\sigma \in S_k} (-q)^{l(\sigma)}
x_{i_1, \alpha_{\sigma (1)}} \cdots x_{i_k, \alpha_{\sigma (k)}},$$ 
where
$I=\{i_1, \dots, i_k\}$, $\Lambda=\{\alpha_1 , \dots , \alpha_k\}$ and
$l(\sigma)$ denotes the length of the 
$k$-permutation $\sigma$. Also, it is convenient
to allow the empty minor: $[\emptyset|\emptyset]_q(M) := 1 \in K$.
Whenever we write a quantum minor in the form $[i_1, \dots, i_k|\alpha_1 , \dots , \alpha_k]_q$, we tacitly assume that the row
and column indices are listed in ascending order, that is, $i_1< \cdots< i_l$
and $\alpha_1< \cdots< \alpha_l$. 

The \emph{quantum minors} in $R$ are the quantum minors of the matrix $(X_{\ia}) \in \mc_{m,p}(R)$. To simplify the notation, we denote by $[I | \Lambda ]_q$ the quantum minor of $R$ associated to the row-index set $I$ and the column-index set $\Lambda$.

It is easy to check that the torus $\hc:=\left( \bbK^\times
\right)^{m+p}$ acts on $R$ by $\bbK$-algebra automorphisms via:
$$(a_1,\dots,a_m,b_1,\dots,b_p).X_{\ia} = a_i b_\alpha X_{\ia} \quad {\rm
for~all} \quad \: (\ia)\in \gc 1,m \dc \times \gc 1,p \dc.$$ 
We refer to this action as the \emph{standard action} of $\left( \bbK^\times
\right)^{m+p}$ on $\mmpc$.
Recall that an \emph{$\hc$-prime ideal} of $R$ is a proper $\hc$-invariant ideal $P$ such that whenever $P$ contains a product $IJ$ of two $\hc$-invariant ideals, it must contain either $I$ or $J$. As $q$ is not a root of unity, it follows from
\cite[5.7]{GoLet2} that there are only finitely many $\hc$-primes in $R$ and that every
$\hc$-prime is completely prime. Hence, the $\hc$-prime ideals of $R$ coincide with the $\hc$-invariant primes. We denote by
$\hspec(R)$ the set of $\hc$-primes of $R$. 

The aim is to parameterise/study the $\hc$-prime ideals in quantum
matrices. 

\begin{example}\label{example:quantumcelld}
The algebra of $2\times 2$ quantum matrices may be presented as 
\[
\oqmtwo := \bbC \left[
\begin{array}{cc}
a & b \\
c &d  \\
\end{array} \right]
\]
with relations   
\[
 ab = qba \qquad ac =qca \qquad bc =cb
\]
\[
bd = qdb \qquad cd = qdc \qquad ad-da = (q-q^{-1})bc.
\]
The {\em quantum determinant} is $D_q:= [12|12]_q=ad-qbc$.

Let $P$ be a prime ideal that contains $d$. Then 
\[
(q-q^{-1})bc= ad-da \in P
\]
and, as $0\neq (q-q^{-1})\in \bbC$ and $P$ is completely prime, we deduce
that either $b\in P$ or $c\in P$. Thus, there is no prime ideal in
$\oqmtwo$ such that $d$ is the only quantum minor that is in $P$. \fin\\

\end{example} 

You should notice the analogy with the corresponding result in the
space of $2\times 2$ totally nonnegative matrices: the cell corresponding to $d$ being the only vanishing minor is empty (see Example \ref{example:celld}). 

\subsection{$\hc$-primes and Cauchon diagrams.}
\label{sectionCauchondiagram}

In \cite{Cau2}, Cauchon showed that his theory of deleting derivations can be applied to the iterated Ore extension $R$. As a consequence, he was able 
to parametrise the set $\hc$-$\mathrm{Spec}(R)$ in terms of combinatorial objects called {\it Cauchon diagrams}.

\begin{definition}\cite{Cau2} {\rm
An $m\times p$ \emph{Cauchon diagram} $C$ is simply an $m\times p$ grid consisting of $mp$ squares in which certain squares are coloured black.  We require that the collection of black squares have the following property: If a square is black, then either every square strictly to its left is black or every square strictly above it is black.

We denote by $\mathcal{C}_{m,p}$ the set of $m\times p$ Cauchon diagrams.
}
\end{definition}

Note that we will often identify an $m \times p $ Cauchon diagram with the set of coordinates of its black boxes. Indeed, if $C \in \mathcal{C}_{m,p}$ and  $(\ia)
 \in \gc 1, m\dc \times \gc 1,  p\dc$, we will say that $(\ia) \in C$ if the box in row $i$ and column $\alpha$ of $C$ is black.

\begin{figure}[h]
\label{fig:CauchonDiagram}
\begin{center}
$$\xymatrixrowsep{0.01pc}\xymatrixcolsep{0.01pc}
\xymatrix{
\plc\edge[0,10]\edge[8,0] &&\plc\edge[8,0] &&\plc\edge[8,0]
&&\plc\edge[8,0] &&\plc\edge[8,0] &&\plc\edge[8,0] \\
 &\mdblk &&\mdblk &&&&\mdblk && &&\plc \\
\plc\edge[0,10] &&&&&&&&&&\plc \\
 &&&&&&&\mdblk && \\
\plc\edge[0,10] &&&&&&&&&&\plc \\
 &\mdblk &&\mdblk &&\mdblk &&\mdblk &&\mdblk \\
\plc\edge[0,10] &&&&&&&&&&\plc \\
 &\mdblk &&\mdblk &&&&\mdblk \\
\plc\edge[0,10] &&\plc &&\plc &&\plc &&\plc &&\plc 
}$$
\caption{An example of a $4\times 5$ Cauchon diagram}
\end{center}
\end{figure}
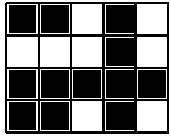

Recall \cite[Corollaire 3.2.1]{Cau2} that Cauchon has constructed (using the deleting derivations algorithm) a bijection between $\hc$-$\mathrm{Spec}(\mmpc)$ and the collection $\mathcal{C}_{m,p}$. As a consequence, Cauchon \cite{Cau2} was able to give a formula for the size of $\hspec (\mmpc)$. This formula was later re-written by Goodearl and McCammond (see \cite{Lau3}) in terms of Stirling numbers of second kind and poly-Bernoulli numbers as defined by Arakawa and Kaneko (see \cite{AK,Kan}).

Notice that the definitions of Le-diagrams and Cauchon diagrams are
the same! Thus, the nonempty cells in totally nonnegative matrices and 
the $\hc$-prime ideals in quantum matrices are parameterised by the
same combinatorial objects. Much more is true, as we will now see. 

For example, $\oqmtwo$ has $14$ $\hc$-prime ideals, as there are $14$ 
Cauchon/Le-diagrams. It is relatively easy to identify these
$\hc$-primes. The following are the $\hc$-prime ideals of $\oqmtwo$.

%%%%%%%%%%%%%%%%%%%%%%%%%%%%%%%%%%%%%%%%%%%%%%%%%%%%%%%%%%%

%%%%%%%%%%%%%%%%%%%%%%%%%%%%%%%%%%%%%%%%%%%%%%%%%%%%%%%%%%%

\[
\xymatrixrowsep{2.4pc}\xymatrixcolsep{4pc}
    
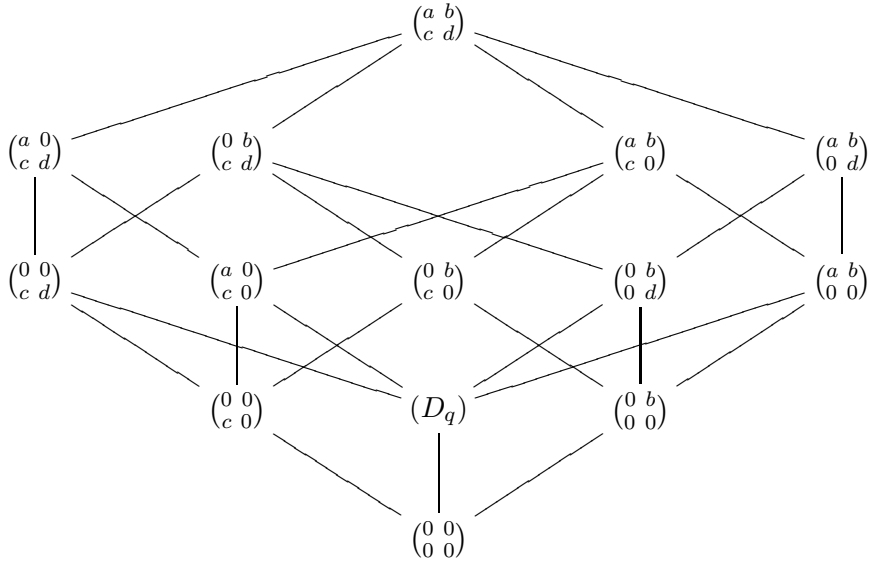
\xymatrix{ &&\tbinom{a~b}{c~d}\\
    \tbinom{a~0}{c~d} \edge[urr]
    &\tbinom{0~b}{c~d} \edge[ur]
    &&\tbinom{a~b}{c~0} \edge[ul]
    &\tbinom{a~b}{0~d} \edge[ull]\\
    \tbinom{0~0}{c~d} \edge[u] \edge[ur]
    &\tbinom{a~0}{c~0} \edge[ul] \edge[urr]
    &\tbinom{0~b}{c~0} \edge[ul] \edge[ur]
    &\tbinom{0~b}{0~d} \edge[ull] \edge[ur]
    &\tbinom{a~b}{0~0} \edge[ul] \edge[u]\\
    &\tbinom{0~0}{c~0} \edge[ul] \edge[u] \edge[ur]
    &\left(D_q\right) \edge[ull] \edge[ul] \edge[ur] \edge[urr]
    &\tbinom{0~b}{0~0} \edge[ul] \edge[u] \edge[ur]\\
    &&\tbinom{0~0}{0~0}  \edge[ul]
    \edge[u] \edge[ur] }
  \]
To interpret this picture, note that, for example, 
$\tbinom{a~b}{c~0}$ denotes the ideal generated by 
$a,b$ and $c$. 

It is easy to check that $13$ of the ideals are prime.
For example, let P be the ideal generated by b and d. Then
$\oqmtwo/P \cong \bbC [a, c]$ 
and $\bbC [a, c]$ is an iterated Ore extension of $\bbC$ and so a domain.
The only problem is to show that the determinant generates a
prime ideal. 
This was originally proved by Jordan, and, independently, by
Levasseur and Stafford. A general result that includes this as a
special case is \cite[Theorem 2.5]{GoLen}.

In fact, in the case that the parameter $q$ is
transcendental over ${\mathbb Q}$, 
the first author, \cite{Lau2} has shown that all
$\ch$-prime ideals are generated by the quantum minors that they contain. In
\cite{GLL2}, this result is extended by replacing $\bbC$ by any field of
characteristic zero (still retaining the condition that $q$ is transcendental
over ${\mathbb Q}$). 
The transcendental restriction is technical: at the moment, 
a certain ideal is 
only known to be prime with this restriction. It is expected that the result 
will remain true when $q$ is merely restricted to be not a root of unity.

If you did Exercise~\ref{exercise-14} then you will notice that the
sets of all quantum minors that define $\hc$-prime ideals in $\oqmtwo$  
are exactly the
quantum versions of the sets of vanishing minors for nonempty cells in 
the space of $2\times 2$ totally nonnegative matrices. This
coincidence also occurs in the general case and an explanation of this
coincidence is obtained in \cite{GLL, GLL2}. However, in order to
explain the coincidence, we need to introduce a third setting, that of 
Poisson matrices, and this is done in the next section.

%%%%%%%%%%%%%%%%%%%%%%%%%%%%%%%%%%%%%%%%%%%%%%%%%%%%%%%%%%%

%%%%%%%%%%%%%%%%%%%%%%%%%%%%%%%%%%%%%%%%%%%%%%%%%%%%%%%%%%%

\section{Poisson matrix varieties and their $\hc$-orbits of symplectic leaves}

%%%%%%%%%%%%%%%%%%%%%%%%%%%%%%%%%%%%%%%%%%

In this section, we study the standard Poisson structure of the
coordinate ring $\OMmpc$ coming from the commutators of 
$\oqmmpc$. Recall that a
{\it Poisson algebra} (over $\bbC$) is a commutative
$\bbC$-algebra $A$ equipped with a Lie bracket $\{-,-\}$ which is a
derivation (for the associative multiplication) in each variable. The
derivations $\{a,-\}$ on $A$ are called {\it Hamiltonian derivations}. When
$A$ is the algebra of complex-valued $C^{\infty}$ functions on a smooth affine
variety $V$, one can use Hamiltonian derivations in order to define
Hamiltonian paths in $V$. A smooth path
$\gamma : [0,1] \rightarrow V$ is a {\it Hamiltonian path in $V$} 
if there exists $H \in C^{\infty}(V)$ such that for all 
$f \in \mathcal{C}^{\infty}(V)$:
\begin{eqnarray}
\label{eq:sympl-leaves}
\frac{d}{dt}(f \circ \gamma )(t)=\{H, f\} \circ \gamma (t),
\end{eqnarray}
for all $0 < t <1$. In other words, Hamiltonian paths are the integral 
curves (or flows) of the Hamiltonian vector fields induced by the 
Poisson bracket. It is easy to check that the relation 
``connected by a piecewise
Hamiltonian path'' 
is an equivalence relation. The equivalence classes of this
relation are called the {\it symplectic leaves} of $V$; they form a partition
of $V$. 

%A {\it Poisson ideal} of $A$ is any ideal $I$ such that $\{A,I\}
%\subseteq I$, and a {\it Poisson prime} ideal is any prime ideal which
%is also a Poisson ideal. The set of Poisson prime ideals in $A$ forms the {\it
%Poisson prime spectrum}, denoted $\mathrm{PSpec} (A)$, which is given the relative Zariski
%topology inherited from $\mathrm{Spec} (A)$. 

\subsection{The Poisson algebra $\OMmpc$}

Denote by $\OMmpc$ the coordinate ring of the variety
$\Mmpc$; note that 
$\OMmpc$ is a (commutative) polynomial algebra in
$mp$ indeterminates $Y_{\ia}$ with $1 \leq i \leq m$ and $1 \leq \alpha \leq
p$. 

The variety $\Mmpc$ is a Poisson variety: there is a unique Poisson bracket on the coordinate ring $\OMmpc $ determined by the following data. For all $(\ia) < (k,\gamma)$, we set: 
$$\{Y_{\ia} ,Y_{k,\gamma} \}=\left\{ \begin{array}{ll}
Y_{\ia}Y_{k,\gamma} &  \mbox{ if } i=k \mbox{ and } \alpha < \gamma \\
Y_{\ia} Y_{k,\gamma} &  \mbox{ if } i< k \mbox{ and } \alpha = \gamma \\
0 &   \mbox{ if } i < k \mbox{ and } \alpha > \gamma \\
2 Y_{i,\gamma} Y_{k,\alpha} &  \mbox{ if } i< k  \mbox{ and } \alpha < \gamma .  \\ 
\end{array} \right.$$
This is the standard Poisson structure on
the affine variety $ \Mmpc$ (cf.~\cite[\S1.5]{BGY}); the Poisson algebra structure on $\OMmpc$ is the semiclassical limit
of the noncommutative algebras $\oqmmpc$. Indeed one can easily check that 
\[
\{ Y_{\ia}, Y_{k,\gamma}\} = \frac{[X_{\ia}, X_{k,\gamma}]}{q-1} \mid _{q=1}.
\]

In particular, the Poisson bracket on $\OMtwoc=\mathbb{C} \left[\begin{array}{cc} a & b \\ c & d \end{array} \right]$ is defined by:
\begin{eqnarray*}
\{a, b\}=ab \hspace{1cm} & \{a, c\}=ac & \hspace{1cm} \{b, c\}=0 \\
 \{b, d\}=bd \hspace{1cm} & \{c, d\}=cd & \hspace{1cm} \{a, d\}=2bc.
\end{eqnarray*}

Note that the Poisson bracket on $\OMmpc$ extends uniquely to a Poisson
bracket on $\mathcal{C}^{\infty}(\Mmpc)$, so that $\Mmpc$ can be viewed as a
Poisson manifold. Hence, $\Mmpc$ can be decomposed as the disjoint union of
its symplectic leaves. Before studying symplectic leaves in $\Mmpc$, let us
explicitly describe the Poisson bracket on $\mathcal{C}^{\infty}(\Mtwoc)$. For
all $f,g \in \mathcal{C}^{\infty}(\Mtwoc)$, one has:
\begin{eqnarray*}
\lefteqn{\{f,g\} =}\\ 
&&ab \left( \frac{\partial f}{\partial a} 
\cdot \frac{\partial g}{\partial b}-\frac{\partial f}{\partial b} 
\cdot \frac{\partial g}{\partial a} \right)
+ac \left( \frac{\partial f}{\partial a} 
\cdot \frac{\partial g}{\partial c}-\frac{\partial f}{\partial c} 
\cdot \frac{\partial g}{\partial a} \right)
+bd \left( \frac{\partial f}{\partial b} 
\cdot \frac{\partial g}{\partial d}-\frac{\partial f}{\partial d} 
\cdot \frac{\partial g}{\partial b} \right)\\
& &+cd \left( \frac{\partial f}{\partial c} 
\cdot \frac{\partial g}{\partial d}-\frac{\partial f}{\partial d} 
\cdot \frac{\partial g}{\partial c} \right)
+2bc \left( \frac{\partial f}{\partial a} 
\cdot \frac{\partial g}{\partial d}-\frac{\partial f}{\partial d} 
\cdot \frac{\partial g}{\partial a} \right).
\end{eqnarray*}

We finish this section by proving an analogue of Examples \ref{example:celld} and \ref{example:quantumcelld} in the Poisson setting.

\begin{proposition} 
Let $\mathcal{L}$ be a symplectic leaf such that $d(M)=0$ for all $M \in
\mathcal{L}$. Then, either $b(M)=0$ for all $M \in \mathcal{L}$ 
or $c(M)=0$ for all $M
\in \mathcal{L}$. 
\end{proposition}

\begin{proof} Let $N=\left( \begin{array}{cc} \alpha & \beta \\ \gamma & 0 \end{array} \right) \in \mathcal{L}$. We first prove that $\beta \gamma=0$. We distinguish between two cases. 

First assume that $\alpha=0$. Then we claim that the path defined by 
$\left( \begin{array}{cc} 0 & \beta \\ 
\gamma & 2\beta\gamma t \end{array} \right)$ 
is a flow of the Hamiltonian vector field associated to $\{a, - \}$, 
and so a Hamiltonian path starting at $N$. 
As $N \in \mathcal{L}$, every point of this Hamiltonian path 
should be in $\mathcal{L}$. In particular, we get that 
$\left( \begin{array}{cc} 0 & \beta \\ 
\gamma & 2\beta\gamma  \end{array} \right) \in \mathcal{L}$. 
As $d(M)=0$ for all $M \in \mathcal{L}$, we get $\beta \gamma =0$ as desired. 

Next assume that $\alpha \neq 0$.  Then we claim that 
$\left( \begin{array}{cc} \alpha & \beta e^{\alpha t} \\ 
\gamma e^{\alpha t} & 
\frac{\beta\gamma}{\alpha}  e^{2\alpha t} \end{array} \right)$ 
is a flow of the Hamiltonian vector field associated to $\{a, - \}$, 
and so a Hamiltonian path starting at $N$. As in the previous case this 
implies $\beta \gamma =0$ as desired. 

Hence, the leaf $\mathcal{L}$ contains a point $N$ of the form 
$$N=\left( \begin{array}{cc} \alpha & \beta \\ 0 & 0 \end{array} \right) \mbox{ or } N=\left( \begin{array}{cc} \alpha & 0 \\ \gamma & 0 \end{array} \right).$$

We prove in the first case that $c(M)=0$ for all $M \in \mathcal{L}$. 
It is enough to prove that if $ \gamma: [0,1] \rightarrow \Mtwoc$ is a Hamiltonian path such that $d(\gamma(t))=0$ for all $t \in [0,1]$ and $c(\gamma(0))=0$, then 
$c(\gamma(t))=0$ for all $t \in [0,1]$. Let $\gamma$ be such a Hamiltonian path. For all $t \in [0,1]$, we set 
$\gamma(t)=  \left( \begin{array}{cc} \gamma_1(t) & \gamma_2(t) \\ \gamma_3(t) & 0 \end{array}\right)$.
It follows from (\ref{eq:sympl-leaves}) that 
\[
\frac{d}{dt}(c \circ \gamma )(t)=\{H, c\} \circ \gamma (t).
\]
Hence as $\gamma_4(t)=0$ for all $t$, we get 
\[ 
\gamma_3'(t) = \gamma_1(t) \gamma_3(t)  \frac{\partial H}{\partial a}(\gamma(t)).
\]
Set $\alpha(t):=\gamma_1(t)   \frac{\partial H}{\partial a}(\gamma(t))$. 
Then 
\[ 
\gamma_3'(t) = \alpha(t) \gamma_3(t) .
\] 
Hence we have 
\[ \gamma_3 (t) = C \exp(\Lambda (t))\]
for all $ t \in [0,1]$, where $C \in \bbC$ and $\Lambda$ is a 
primitive of $\alpha \in C^{\infty}([0,1])$. As $\gamma_3(0)=0$ we must 
have $C=0$, so that $\gamma_3(t)=0$ for all $t$, as desired.
\end{proof}

\subsection{$\hc$-orbits of symplectic leaves in $\Mmpc$}

Notice that the torus $\hc:=\left( \bbC^\times \right)^{m+p}$ acts on
$\OMmpc$ by Poisson automorphisms via: $$(a_1,\dots,a_m,b_1,\dots,b_p).Y_{\ia}
= a_i b_\alpha Y_{\ia} \quad {\rm for~all} \quad \: (\ia)\in \gc 1,m \dc
\times \gc 1,p \dc.$$ 
Note that $\hc$ is acting rationally on $\OMmpc$.

At the geometric level, this action of the algebraic torus $\hc$ comes from the left action of $\hc$ on $\Mmpc$ by Poisson isomorphisms via:
$$(a_1,\dots,a_m,b_1,\dots,b_p).M :=\mbox{diag}(a_1, \dots, a_m)^{-1} \cdot  M \cdot \mbox{diag}(b_1, \dots, b_p)^{-1}.$$
This action of $\hc$ on $\Mmpc$ induces an action of
$\hc$ on the set $\mathrm{Sympl}(\Mmpc)$ of symplectic
leaves in $\Mmpc$ (cf.~\cite[\S0.1]{BGY}). As in \cite{BGY}, we view the
$\hc$-orbit of a symplectic leaf $\lc$ as the set-theoretic union
$\bigcup_{h\in\hc} h.\lc \subseteq \Mmpc$, rather than as the family $\{h.\lc \mid
h\in\hc\}$. We denote the set of such orbits by
$\hc$-$\mathrm{Sympl}(\Mmpc)$.

As the symplectic leaves of $\Mmpc$ form a partition of $\Mmpc$, so too 
do the
$\hc$-orbits of symplectic leaves.

\begin{example} The symplectic leaf $\mathcal{L}$ containing $\left( \begin{array}{cc}1 & 1 \\ 1 & 1 \end{array}\right)$ is the set $\mathcal{E}$ of those $2 \times 2$ complex matrices $M=\left( \begin{array}{cc} x & y \\ z & t \end{array}\right)$ with $y-z=0$, $xt-yz=0$ and $y\neq 0$. In other words, 
\[
\mathcal{E}:= \{ M \in \Mtwoc ~|~ \Delta (M)= 0 \mbox{, } (b-c)(M)=0 
\mbox{ and } b(M) \neq 0\},
\] 
where $a,b,c,d$ denote the canonical generators of the coordinate ring of $\Mtwoc$ and $\Delta:= ad-bc$ is the determinant function. It easily follows from this that the $\hc$-orbit of symplectic leaves in $\Mtwoc$ that contains the point $\left( \begin{array}{cc} 1 & 1 \\ 1 & 1 \end{array} \right)$ is the set of those $2 \times 2$ matrices $M$
with $\Delta(M)=0$ and $b(M).c(M) \neq 0$. Moreover the closure of this $\mathcal{H}$-orbit coincides with the set of those $2 \times 2$ matrices $M$ with $\Delta(M)=0$. 
\fin\\
\end{example}

The $\hc$-orbits of symplectic leaves in $\Mmpc$ have been explicitly 
described by Brown, Goodearl and Yakimov in \cite{BGY}. The following result was proved in \cite[Theorems 3.9, 3.13, 4.2]{BGY}.
\begin{theorem}
Set $$\mathcal{S}= \{w \in
S_{m+p} \ | \ -p \leq w(i) -i \leq m \mbox{ for all }i=1,2, \dots, m+p\} . $$
\begin{enumerate}
\item The $\hc$-orbits of symplectic leaves in
$\mc_{m,p}(\bbC)$ are smooth irreducible locally
closed subvarieties. 
\item There is an explicit $1 : 1$ correspondence between $\mathcal{S}$ and $\hc$-$\mathrm{Sympl}(\mathcal{M}_{m,p}(\mathbb{C}))$. 
\item Each $\hc$-orbit is defined by some rank conditions.
\end{enumerate}
\end{theorem}

The rank conditions that define the $\hc$-orbits of symplectic leaves and their closures are explicit in \cite{BGY}. The reader is  refered to \cite{BGY} for more details.

For $w \in \mathcal{S}$, we denote by $\mathcal{P}_{w}$ the $\hc$-orbit
of symplectic leaves associated to the restricted permutation $w$.

Before going any further let us look at a special case. In the $2 \times 2$ case, the Theorem of Brown, Goodearl and Yakimov asserts that there is a $1:1$ correspondence between 
\[
\mathcal{S}= \{w \in
S_{4} \ | \ -2 \leq w(i) -i \leq 2 \mbox{ for all }i=1,2,3,4\}.
\] and $\hc$-$\mathrm{Sympl}(\mathcal{M}_{2}(\mathbb{C}))$. 
In other words, there is a bijection between the set of those permutations $w$ in $S_4$ such that $w(1) \neq 4$ and $w(4) \neq 1$. 
One may be disappointed not to retrieve $2 \times 2 $ Cauchon diagrams, but a direct inspection shows that there are exactly 14 such restricted permutations in the $2 \times 2$ case! This is not at all a coincidence as we will see in the following section. 

To finish, let us mention that the set of all minors that vanish 
on the closure of the $\hc$-orbit of symplectic leaves associated to 
$w \in \mathcal{S}$ has been described in \cite{GLL}. 
In order to describe this result, we need to introduce some notation. 

Set $N = m+p$, and let $\wom$, $\wop$ and $\woN$ denote the
longest elements in $S_m$, $S_p$ and $S_N$, respectively, so that 
$w_\circ^r(i) = r + 1 - i$ for $i = 1, . . . , r$.

\begin{definition}\label{def:mw}
For $w\in \mathcal{S}$, define $\mathcal{M}(w)$ to be the set of minors $[I|\Lambda]$, with
$I\subseteq \gc 1,  m\dc$ and $\Lambda\subseteq \gc 1, p\dc$, that satisfy at least one of the following conditions. 
\begin{enumerate}
\item $I\not\leq \wom w(L)$ for all $L\subseteq \onep\cap
w^{-1} \onem$ such that $|L|=|I|$ and $L\leq \Lambda$. 
\item $m+\Lambda\not\leq w\woN(L)$ for all $L\subseteq \onem\cap
\woN w^{-1} \gc m+1, N\dc$ such that $|L|=|\Lambda|$ and $L\leq I$. 
\item There exist $1\leq r \leq s\leq p$ and 
$\Lambda' \subseteq \Lambda\cap \gc r, s \dc$ such that\\ 
$|\Lambda'|> |\gc r, s \dc\setminus w^{-1} \gc m+r,\, m+s\dc |$.
\item There exist $1\leq r \leq s\leq m$ and $I'\subseteq I\cap \gc r, s \dc$
such that\\ 
$|I'|>|\woN \gc r, s \dc \setminus w^{-1} \wom \gc r, s \dc |$.
\end{enumerate}
\end{definition}

\begin{example}
For example, when $m=p=3$ and $w=(2~3~5~4)$, then $\mathcal{M}(w) =\{[1,2|2,3],[1,3|2,3],[2,3|2,3],[2,3|1,3],[2,3|1,2],[1,2,3|1,2,3]\}$. 
We observe that this family of minors defines a nonempty cell in $\mathcal{M}^{\mathrm{tnn}}_3(\bbR)$ by Example \ref{example:celld}. 
\fin\\
\end{example}

In \cite{GLL}, the following result was obtained thanks to previous results of \cite{BGY} and \cite{Ful}.
\begin{theorem}\label{theorem-4.2reformulation} 
Let $w\in \mathcal{S}$. The closure of the $\hc$-orbit $\mathcal{P}_{ w}$ is given by: 
\[
\overline{\mathcal{P}_{ w}}
=\{x\in \Mmpc\mid [I|J](x)=0~{\rm for~all}~[I|J]\in\mathcal{M}(w)\}.
\]
Moreover, the minor $[I|J]$ vanishes on $\overline{\mathcal{P}_{ w}}$ 
if and only if $[I|J]\in\mathcal{M}(w)$.
\end{theorem}

\section{From Cauchon diagrams to restricted permutations and back, via pipe dreams} 

In the previous section, we have seen that the torus-orbits of symplectic
leaves in $\mathcal{M}_{m,p}(\mathbb{C})$ are parameterised by the 
restricted permutations in $S_{m+p}$ given by 
\[
\mathcal{S}= \{w \in
S_{m+p} \ | \ -p \leq w(i) -i \leq m \mbox{ for all }i=1,2, \dots, m+p\}.
\]

In the $2\times 2$ case, this subposet of the Bruhat poset of $S_4$ is 
\[
\mathcal{S}= \{w \in
S_{4} \ | \ -2 \leq w(i) -i \leq 2 \mbox{ for all }i=1,2,3,4\}.
\]
and is shown 
below. 
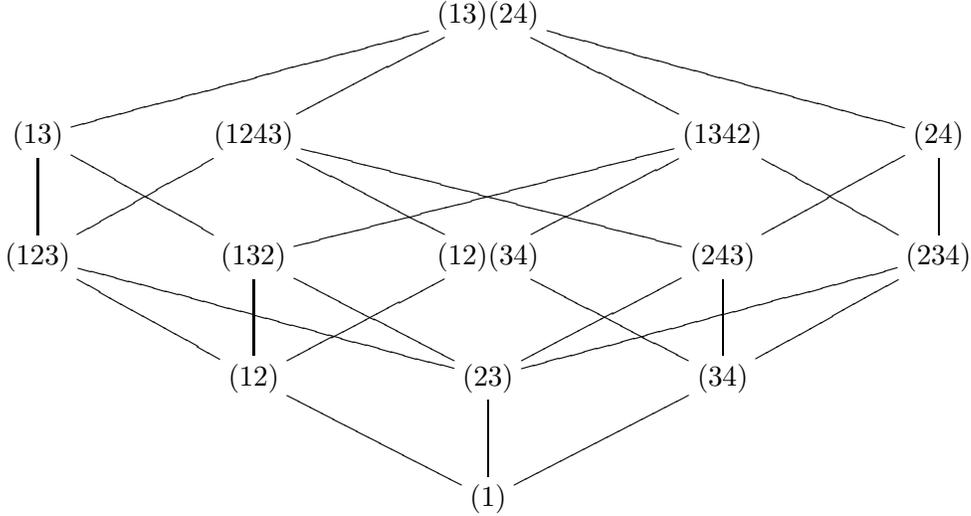
\[
\xymatrixrowsep{2.4pc}\xymatrixcolsep{4pc}
    \xymatrix{ &&(13)(24)\\
    (13) \edge[urr]
    &(1243) \edge[ur]
    &&(1342) \edge[ul]
    &(24) \edge[ull]\\
    (123) \edge[u] \edge[ur]
    &(132) \edge[ul] \edge[urr]
    &(12)(34) \edge[ul] \edge[ur]
    &(243) \edge[ull] \edge[ur]
    &(234) \edge[ul] \edge[u]\\
    &(12) \edge[ul] \edge[u] \edge[ur]
    &(23) \edge[ull] \edge[ul] \edge[ur] \edge[urr]
    &(34) \edge[ul] \edge[u] \edge[ur]\\
    &&(1)  \edge[ul]
    \edge[u] \edge[ur] }
  \]
Inspection of this poset reveals that it is isomorphic to the poset of the 
$\ch$-prime ideals of $\oqmtwo$ displayed in Section 2; and so to a similar 
poset of the Cauchon diagrams corresponding to the $\ch$-prime ideals. 

More generally, it is known that the numbers of $\hc$-primes in $\mmpc$ (and so the number of $m \times p$ Cauchon diagrams) is equal to $|\mathcal{S}|$ (see \cite{Lau4}). This is no coincidence, and the connection between the two posets can be illuminated by using {\em Pipe Dreams}. 

The procedure to produce a restricted permutation from a Cauchon diagram goes
as follows. Given a Cauchon diagram, replace each  
black box by a cross, and each white box by an elbow joint, that is:

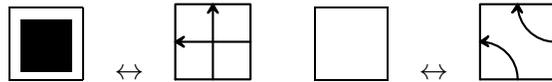
\begin{center}
\begin{tabular}{ccccccc}
{\Huge \young(\bbox)}   & $\leftrightarrow$ & \begin{pgfpicture}{0.5cm}{0.5cm}{1.5cm}{1.5cm}%
\pgfsetroundjoin \pgfsetroundcap%
%Carre  (Ligne Polyg.)
\pgfsetlinewidth{0.8pt} 
\pgfmoveto{\pgfxy(0.5,0.5)}\pgflineto{\pgfxy(1.5,0.5)}\pgflineto{\pgfxy(1.5,1.5)}\pgflineto{\pgfxy(0.5,1.5)}\pgflineto{\pgfxy(0.5,0.5)}\pgfstroke
%fleche1  (Ligne Polyg.)
\pgfsetlinewidth{0.8pt} 
\pgfxyline(1,1.5)(1,0.5)
\pgfsetfillcolor{black}
\pgfmoveto{\pgfxy(1.0667,1.3845)}\pgflineto{\pgfxy(1,1.5)}\pgflineto{\pgfxy(0.9333,1.3845)}\pgflineto{\pgfxy(1,1.4423)}\pgfclosepath\pgffillstroke
%fleche2  (Ligne Polyg.)
\pgfxyline(0.5,1)(1.5,1)
\pgfmoveto{\pgfxy(0.6155,1.0667)}\pgflineto{\pgfxy(0.5,1)}\pgflineto{\pgfxy(0.6155,0.9333)}\pgflineto{\pgfxy(0.5577,1)}\pgfclosepath\pgffillstroke
\end{pgfpicture} & \qquad &
 {\Huge \young(\wbox)}  & $\leftrightarrow$ &
\begin{pgfpicture}{0.5cm}{0.5cm}{1.5cm}{1.5cm}%
\pgfsetlinewidth{0.8pt} 
\pgfsetroundjoin \pgfsetroundcap%
%objet2  (Arc)

\pgfmoveto{\pgfxy(0.5,1)}\pgfpatharc{90}{0}{0.5cm}
\pgfstroke
\pgfsetfillcolor{black}
\pgfmoveto{\pgfxy(1.0601,1.3884)}\pgflineto{\pgfxy(1,1.5)}\pgflineto{\pgfxy(0.927,1.3884)}\pgflineto{\pgfxy(0.9968,1.4424)}\pgfclosepath\pgffillstroke
%objet4  (Ligne Polyg.)
\pgfmoveto{\pgfxy(0.5,0.5)}\pgflineto{\pgfxy(1.5,0.5)}\pgflineto{\pgfxy(1.5,1.5)}\pgflineto{\pgfxy(0.5,1.5)}\pgflineto{\pgfxy(0.5,0.5)}\pgfstroke
%objet5  (Arc)
\pgfmoveto{\pgfxy(1,1.5)}\pgfpatharc{180}{270}{0.5cm}
\pgfstroke
\pgfmoveto{\pgfxy(0.6116,1.073)}\pgflineto{\pgfxy(0.5,1)}\pgflineto{\pgfxy(0.6116,0.9399)}\pgflineto{\pgfxy(0.5576,1.0032)}\pgfclosepath\pgffillstroke
\end{pgfpicture}
\end{tabular}
\end{center}

For example, the Cauchon diagram 
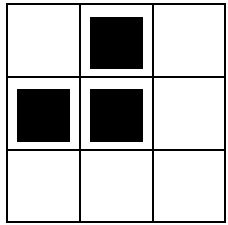
\begin{center}
{\Huge \young(\wbox\bbox\wbox,\bbox\bbox\wbox,\wbox\wbox\wbox)}
\end{center}
produces the pipe dream 

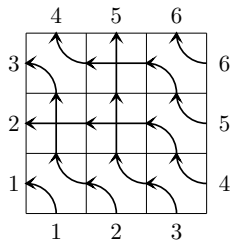
\begin{center}\begin{minipage}{.3\linewidth}
\scalebox{0.8}{\begin{pgfpicture}{-0.5cm}{-0.5cm}{4.5cm}{4.5cm}%
%objet1  (Ligne Polyg.)
\pgfsetlinewidth{0.2pt} 
\pgfxyline(0.5,3.5)(3.5,3.5)
%objet2  (Ligne Polyg.)
\pgfxyline(0.5,2.5)(3.5,2.5)
%objet3  (Ligne Polyg.)
\pgfxyline(0.5,1.5)(3.5,1.5)
%objet4  (Ligne Polyg.)
\pgfxyline(0.5,0.5)(3.5,0.5)
%objet5  (Ligne Polyg.)
\pgfxyline(3.5,0.5)(3.5,3.5)
%objet6  (Ligne Polyg.)
\pgfxyline(2.5,0.5)(2.5,3.5)
%objet7  (Ligne Polyg.)
\pgfxyline(1.5,0.5)(1.5,3.5)
%objet8  (Ligne Polyg.)
\pgfxyline(0.5,0.5)(0.5,3.5)
%objet9  (Ligne Polyg.)
%objet10  (Label)
\pgfputat{\pgfxy(0.3,1)}{\pgfnode{rectangle}{center}{\color{black}\small 1}{}{\pgfusepath{}}}
%objet11  (Label)
\pgfputat{\pgfxy(0.3,2)}{\pgfnode{rectangle}{center}{\color{black}\small 2}{}{\pgfusepath{}}}
%objet12  (Label)
\pgfputat{\pgfxy(0.3,3)}{\pgfnode{rectangle}{center}{\color{black}\small 3}{}{\pgfusepath{}}}
%objet13  (Label)
\pgfputat{\pgfxy(3.8,1)}{\pgfnode{rectangle}{center}{\color{black}\small 4}{}{\pgfusepath{}}}
%objet14  (Label)
\pgfputat{\pgfxy(3.8,2)}{\pgfnode{rectangle}{center}{\color{black}\small 5}{}{\pgfusepath{}}}
%objet15  (Label)
\pgfputat{\pgfxy(3.8,3)}{\pgfnode{rectangle}{center}{\color{black}\small 6}{}{\pgfusepath{}}}
%objet16  (Label)
\pgfputat{\pgfxy(1,0.2)}{\pgfnode{rectangle}{center}{\color{black}\small 1}{}{\pgfusepath{}}}
%objet17  (Label)
\pgfputat{\pgfxy(2,0.2)}{\pgfnode{rectangle}{center}{\color{black}\small 2}{}{\pgfusepath{}}}
%objet18  (Label)
\pgfputat{\pgfxy(3,0.2)}{\pgfnode{rectangle}{center}{\color{black}\small 3}{}{\pgfusepath{}}}
%objet19  (Label)
\pgfputat{\pgfxy(1,3.8)}{\pgfnode{rectangle}{center}{\color{black}\small 4}{}{\pgfusepath{}}}
%objet20  (Label)
\pgfputat{\pgfxy(2,3.8)}{\pgfnode{rectangle}{center}{\color{black}\small 5}{}{\pgfusepath{}}}
%objet21  (Label)
\pgfputat{\pgfxy(3,3.8)}{\pgfnode{rectangle}{center}{\color{black}\small 6}{}{\pgfusepath{}}}

%BOITE(3,1)
\pgfsetlinewidth{0.8pt} 
\pgfmoveto{\pgfxy(1,1.5)}\pgfpatharc{180}{270}{0.5cm}
\pgfstroke
\pgfsetfillcolor{black}
\pgfmoveto{\pgfxy(1.0601,1.3845)}\pgflineto{\pgfxy(1,1.5)}\pgflineto{\pgfxy(0.927,1.3845)}\pgflineto{\pgfxy(0.9968,1.4425)}\pgfclosepath\pgffillstroke
%%%%%%%%%%%%%%%%%%%%%%%%%%%%%%
\pgfmoveto{\pgfxy(0.5,1)}\pgfpatharc{90}{0}{0.5cm}
\pgfstroke
\pgfstroke
\pgfmoveto{\pgfxy(0.6155,1.073)}\pgflineto{\pgfxy(0.5,1)}\pgflineto{\pgfxy(0.6155,0.9399)}\pgflineto{\pgfxy(0.5577,1.0032)}\pgfclosepath\pgffillstroke

%BOITE(3,2)
\pgfsetlinewidth{0.8pt} 
\pgfmoveto{\pgfxy(2,1.5)}\pgfpatharc{180}{270}{0.5cm}
\pgfstroke
\pgfsetfillcolor{black}
\pgfmoveto{\pgfxy(2.0601,1.3845)}\pgflineto{\pgfxy(2,1.5)}\pgflineto{\pgfxy(1.927,1.3845)}\pgflineto{\pgfxy(1.9968,1.4425)}\pgfclosepath\pgffillstroke
%%%%%%%%%%%%%%%%%%%%%%%%%%%%%%%%
\pgfmoveto{\pgfxy(1.5,1)}\pgfpatharc{90}{0}{0.5cm}
\pgfstroke
\pgfstroke
\pgfmoveto{\pgfxy(1.6155,1.073)}\pgflineto{\pgfxy(1.5,1)}\pgflineto{\pgfxy(1.6155,0.9399)}\pgflineto{\pgfxy(1.5577,1.0032)}\pgfclosepath\pgffillstroke

%BOITE(3,3)
\pgfsetlinewidth{0.8pt} 
\pgfmoveto{\pgfxy(3,1.5)}\pgfpatharc{180}{270}{0.5cm}
\pgfstroke
\pgfsetfillcolor{black}
\pgfmoveto{\pgfxy(3.0601,1.3845)}\pgflineto{\pgfxy(3,1.5)}\pgflineto{\pgfxy(2.927,1.3845)}\pgflineto{\pgfxy(2.9968,1.4425)}\pgfclosepath\pgffillstroke
%%%%%%%%%%%%%%%%%%%%%%%%%%%%%%%%
\pgfmoveto{\pgfxy(2.5,1)}\pgfpatharc{90}{0}{0.5cm}
\pgfstroke
\pgfstroke
\pgfmoveto{\pgfxy(2.6155,1.073)}\pgflineto{\pgfxy(2.5,1)}\pgflineto{\pgfxy(2.6155,0.9399)}\pgflineto{\pgfxy(2.5577,1.0032)}\pgfclosepath\pgffillstroke

%BOITE(2,3)
\pgfsetlinewidth{0.8pt} 
\pgfmoveto{\pgfxy(3,2.5)}\pgfpatharc{180}{270}{0.5cm}
\pgfstroke
\pgfsetfillcolor{black}
\pgfmoveto{\pgfxy(3.0601,2.3845)}\pgflineto{\pgfxy(3,2.5)}\pgflineto{\pgfxy(2.927,2.3845)}\pgflineto{\pgfxy(2.9968,2.4425)}\pgfclosepath\pgffillstroke
%%%%%%%%%%%%%%%%%%%%%%%%%%%%%%%%
\pgfmoveto{\pgfxy(2.5,2)}\pgfpatharc{90}{0}{0.5cm}
\pgfstroke
\pgfstroke
\pgfmoveto{\pgfxy(2.6155,2.073)}\pgflineto{\pgfxy(2.5,2)}\pgflineto{\pgfxy(2.6155,1.9399)}\pgflineto{\pgfxy(2.5577,2.0032)}\pgfclosepath\pgffillstroke

%BOITE(1,3)
\pgfsetlinewidth{0.8pt} 
\pgfmoveto{\pgfxy(3,3.5)}\pgfpatharc{180}{270}{0.5cm}
\pgfstroke
\pgfsetfillcolor{black}
\pgfmoveto{\pgfxy(3.0601,3.3845)}\pgflineto{\pgfxy(3,3.5)}\pgflineto{\pgfxy(2.927,3.3845)}\pgflineto{\pgfxy(2.9968,3.4425)}\pgfclosepath\pgffillstroke
%%%%%%%%%%%%%%%%%%%%%%%%%%%%%%%%
\pgfmoveto{\pgfxy(2.5,3)}\pgfpatharc{90}{0}{0.5cm}
\pgfstroke
\pgfstroke
\pgfmoveto{\pgfxy(2.6155,3.073)}\pgflineto{\pgfxy(2.5,3)}\pgflineto{\pgfxy(2.6155,2.9399)}\pgflineto{\pgfxy(2.5577,3.0032)}\pgfclosepath\pgffillstroke

%BOITE(1,1)
\pgfsetlinewidth{0.8pt} 
\pgfmoveto{\pgfxy(1,3.5)}\pgfpatharc{180}{270}{0.5cm}
\pgfstroke
\pgfsetfillcolor{black}
\pgfmoveto{\pgfxy(1.0601,3.3845)}\pgflineto{\pgfxy(1,3.5)}\pgflineto{\pgfxy(0.927,3.3845)}\pgflineto{\pgfxy(0.9968,3.4425)}\pgfclosepath\pgffillstroke
%%%%%%%%%%%%%%%%%%%%%%%%%%%%%%%%
\pgfmoveto{\pgfxy(0.5,3)}\pgfpatharc{90}{0}{0.5cm}
\pgfstroke
\pgfstroke
\pgfmoveto{\pgfxy(0.6155,3.073)}\pgflineto{\pgfxy(0.5,3)}\pgflineto{\pgfxy(0.6155,2.9399)}\pgflineto{\pgfxy(0.5577,3.0032)}\pgfclosepath\pgffillstroke

%BOITE (1,2)
\pgfxyline(1.5,3)(2.5,3)
\pgfmoveto{\pgfxy(1.6155,3.0667)}\pgflineto{\pgfxy(1.5,3)}\pgflineto{\pgfxy(1.6155,2.9333)}\pgflineto{\pgfxy(1.5577,3)}\pgfclosepath\pgffillstroke
%%%%%%%%%%%%%%%%%%%%%%%%%%%%%%%%%%%%%
\pgfxyline(2,2.5)(2,3.5)
\pgfmoveto{\pgfxy(2.0667,3.3845)}\pgflineto{\pgfxy(2,3.5)}\pgflineto{\pgfxy(1.9333,3.3845)}\pgflineto{\pgfxy(2,3.4425)}\pgfclosepath\pgffillstroke

%BOITE (2,2)
\pgfxyline(1.5,2)(2.5,2)
\pgfmoveto{\pgfxy(1.6155,2.0667)}\pgflineto{\pgfxy(1.5,2)}\pgflineto{\pgfxy(1.6155,1.9333)}\pgflineto{\pgfxy(1.5577,2)}\pgfclosepath\pgffillstroke
%%%%%%%%%%%%%%%%%%%%%%%%%%%%%%%%%%%%%
\pgfxyline(2,1.5)(2,2.5)
\pgfmoveto{\pgfxy(2.0667,2.3845)}\pgflineto{\pgfxy(2,2.5)}\pgflineto{\pgfxy(1.9333,2.3845)}\pgflineto{\pgfxy(2,2.4425)}\pgfclosepath\pgffillstroke

%BOITE (2,1)
\pgfxyline(0.5,2)(1.5,2)
\pgfmoveto{\pgfxy(0.6155,2.0667)}\pgflineto{\pgfxy(0.5,2)}\pgflineto{\pgfxy(0.6155,1.9333)}\pgflineto{\pgfxy(0.5577,2)}\pgfclosepath\pgffillstroke
%%%%%%%%%%%%%%%%%%%%%%%%%%%%%%%%%%%%%
\pgfxyline(1,1.5)(1,2.5)
\pgfmoveto{\pgfxy(1.0667,2.3845)}\pgflineto{\pgfxy(1,2.5)}\pgflineto{\pgfxy(0.9333,2.3845)}\pgflineto{\pgfxy(1,2.4425)}\pgfclosepath\pgffillstroke

\end{pgfpicture}}
\end{minipage}
\end{center}

We obtain a permutation $\sigma$ 
from the pipe dream in the following way. To calculate $\sigma(i)$, locate 
the $i$ either on the right hand side or the bottom of the pipe dream and 
and trace through the pipe dream to find the number $\sigma(i)$ that is at the end of the pipe starting at $i$. In the example displayed, 
we find that $\sigma= 135246$ (in one-line notation). 

It is easy to check that this produces a restricted permutation of the
required type by using the observation that as you move along a pipe from
source to image, you can only move upwards and leftwards; so, for example, in
any $3\times 3$ example $\sigma(2)$ is at most $5$ (the number directly above $2$). 

This procedure provides an explicit bijection between the set of $m \times p$ Cauchon diagrams and the poset $\mathcal{S}$ (see \cite{Pos,CaMe}).

\section{The Unifying Theory} 

In the previous sections we have seen that the nonempty cells in $\mtnnmp$, the torus-invariant prime ideals in $\mmpc$ and the closure of the $\hc$-orbits of symplectic leaves are all parametrised by $m \times p$ Cauchon diagrams. This suggests that there might be a connection between these objects. Going a step further, all these objects are characterised by certain families of (quantum) minors. 

First, totally nonnegative cells are defined by the vanishing of families of
minors. Some of the TNN cells are empty. So it is natural to introduce the
following definition. A family of minors is {\it admissible} if the
corresponding TNN cell is nonempty. The obvious question to ask is:\\

\noindent 
{\bf Question: what are the admissible families of minors?}\\

Next, in the quantum case, $\hc$-primes of $\mmpc$ are generated by quantum
minors when we assume that $q$ is transcendental over $\mathbb{Q}$. The
obvious question in this setting is:\\

\noindent 
{\bf Question: which families of quantum minors generate 
$\hc$-invariant prime ideals?}\\

Finally, it follows from the work of Brown, Goodearl and Yakimov that the
closure of the $\hc$-orbits of symplectic leaves in $\Mmpc$ are algebraic, and
are defined by rank conditions. In other words, they are defined by the
vanishing of some families of minors. The obvious question in this context is:
\\ 

\noindent 
{\bf Question: which families of minors?}\\

At first, we may be tempted to propose the following conjecture. 
Let $Z_q$ be a family of quantum minors, and $Z$ be the corresponding family of minors. Then $\langle Z_q \rangle $ is a $\hc$-prime ideal if and only if the cell $S^0_{Z}$ is nonempty.

Stated like this, this conjecture is wrong. The problem here is that distinct families of minors may generate the same $\hc$-invariant prime ideal.
For instance, the ideal generated by $a$ and $b$ in $\oqmtwo$ coincides with
the ideal generated by $a$, $b$ and the quantum determinant $[1,2|1,2]_q$;
moreover this ideal is an $\hc$-invariant prime ideal. So we need to be a bit
more precise in order to get a correct statement. It turns out that the right
thing to do is to compare the admissible families of minors first with the set
of all minors that vanish on the closure of a torus-orbit of symplectic leaves in
$\Mmpc$, and second with the set of all quantum minors that belong to a
torus-invariant prime ideal in $\mmpc$.

\subsection{An algorithm to rule them all}

In \cite{Cau,Cau1,Cau2}, Cauchon developed and used an algorithm, called the
{\it deleting derivations algorithm} in order to study the $\hc$-invariant
prime ideals in $\mmpc$. Roughly speaking, in the $2 \times 2$ case, this
algorithm consists in the following change of variable: 

$$\left( \begin{array}{cc}
a & b \\ c & d 
\end{array}\right) \longrightarrow 
\left( \begin{array}{cc}
a-bd^{-1}c & b \\ c & d 
\end{array}\right) .$$
Let us now give a precise definition of the deleting derivations algorithm.  

If $M=(x_{i,\alpha})\in \mathcal{M}_{m,p}(K)$, then we set 
$$g_{j,\beta}(M)=(x'_{i,\alpha}) \in \mathcal{M}_{m,p}(K),$$
where
$$x'_{i,\alpha} := \left\{ \begin{array}{ll}
x_{i,\alpha} - x_{i,\beta}x_{j,\beta}^{-1} x_{j,\alpha} &  \mbox{if } x_{j,\beta} \neq 0, \: i < j \mbox{ and } \alpha < \beta \\ 
x_{i,\alpha}  &  \mbox{otherwise. }
\end{array} \right. $$

We set $M^{(j,\beta)}:= g_{j,\beta} \circ  \dots  \circ  g_{m,p-1}  \circ g_{m,p}(M)$ where the indices are taken in  lexicographic order.  

The matrix $M^{(1,1)}$ is called the matrix obtained from $M$ at the end of
the deleting derivations algorithm.

The deleting derivations algorithm has an inverse that is called the 
{\em restoration algorithm}. It was originally developed in \cite{Lau} 
to study
$\hc$-primes in quantum matrices. 
Roughly speaking, in the $2 \times 2$ case, 
the restoration algorithm consists of making the following change of variable:
$$\left( \begin{array}{cc}
a & b \\ c & d 
\end{array}\right) \longrightarrow 
\left( \begin{array}{cc}
a+bd^{-1}c & b \\ c & d 
\end{array}\right) .$$
Let us now give a precise definition of the restoration algorithm.  

If $M=(x_{i,\alpha})\in \mathcal{M}_{m,p}(K)$, then we set 
$$f_{j,\beta}(M)=(x'_{i,\alpha}) \in \mathcal{M}_{m,p}(K),$$
where
$$x'_{i,\alpha} := \left\{ \begin{array}{ll}
x_{i,\alpha} + x_{i,\beta}x_{j,\beta}^{-1} x_{j,\alpha} &  \mbox{if } x_{j,\beta} \neq 0, \: i < j \mbox{ and } \alpha < \beta \\ 
x_{i,\alpha}  &  \mbox{otherwise. }
\end{array} \right. $$

We set $M^{(j,\beta)}:=f_{j,\beta} \circ \dots \circ f_{1,2}  \circ f_{1,1}(M)$ where the indices are taken in the reverse of the lexicographic order.  

The matrix $M^{(m,p)}$ is called the matrix obtained from $M$ at the end of
the restoration algorithm.

\begin{example}
\label{exampleRestorationAlgo}
Set $M=\left( \begin{array}{rrr} 
 1 & -1 & 1\\
 0 & 2 & 1\\
 1 & 1 & 1 
 \end{array}\right)$. Then, applying the restoration algorithm to $M$, we get successively: 
 $$M^{(2,2)}=M^{(2,1)}=M^{(1,3)}=M^{(1,2)}=M^{(1,1)}=M,$$
 \ 
$$M^{(3,1)}=M^{(2,3)}= \left( \begin{array}{rrr} 
 1 & 1 & 1\\
 0 & 2 & 1\\
 1 & 1 & 1 
 \end{array}\right), \ \ \ M^{(3,2)}= \left( \begin{array}{rrr} 
 2 & 1 & 1\\
 2 & 2 & 1\\
 1 & 1 & 1 
 \end{array}\right)$$
 and 
 $$M^{(3,3)}= \left( \begin{array}{rrr} 
 3 & 2 & 1\\
 3 & 3 & 1\\
 1 & 1 & 1 
 \end{array}\right).$$ 
which is the matrix obtained from $M$ at the end of the restoration 
algorithm. \fin\\ 
\end{example}

\subsection{The restoration algorithm and TNN matrices}

It is easy to see that the matrix $M^{(3,3)}$ obtained from $M$ by the restoration
algorithm in Example \ref{exampleRestorationAlgo} is not TNN. In fact, the
only minor that is negative is $[1,2|2,3](M^{(3,3)})$. The reason for this failure to be
TNN is that the starting matrix $M$ has a negative entry. Moreover one can
check by following the steps of the restoration algorithm that
$[1,2|2,3](M^{(3,3)})=m_{1,2}m_{2,3}$. In general, one can express the
(quantum) minors of $M^{(j,\beta)^+}$ in terms of the (quantum) minors of
$M^{(j,\beta)}$ (see \cite{GLL,GLL2}). As a consequence, one is able to prove the following
result that gives a necessary and sufficient condition for a real matrix to be
TNN.

\begin{theorem} \cite{GLL}
\begin{enumerate}
\item If the entries of $M$ are nonnegative and its zeros form a Cauchon
diagram, then the matrix $M^{(m,p)}$ obtained from $M$ at the end of
the restoration algorithm is TNN. 
\item Let $M$ be a matrix with real entries.
We can apply the deleting derivations algorithm to $M$. Let $N$ denote the
matrix obtained at the end of the deleting derivations algorithm.\\ 
Then $M$ is TNN if and only if the matrix $N$ is nonnegative and
its zeros form a Cauchon diagram. (That is, the zeros of $N$ correspond to the black boxes of a Cauchon diagram.)
\end{enumerate}
\end{theorem}

\begin{exercise}\rm  
Use the deleting derivations algorithm to test whether the following matrices are TNN:
$$ M_1=\left( \begin{array}{rrrr} 
11 & 7 & 4 & 1 \\
7  & 5 & 3 & 1 \\
4  & 3 & 2 & 1 \\ 
1  & 1 & 1 & 1
 \end{array}\right) \mbox{ and } M_2=\left( \begin{array}{rrrr} 
7 & 5 & 4 & 1 \\
6 & 5 & 3 & 1 \\
4 & 3 & 2 & 1 \\ 
1 & 1 & 1 & 1
 \end{array}\right).$$
\end{exercise}

 \subsection{Main result}

Let $C$ be an $m \times p$ Cauchon diagram and $T=(t_{i,\alpha})$ be a matrix with entries in a skew-field $K$. Assume that $t_{i,\alpha}=0$ if and only if $(i,\alpha)$ is a black box of $C$. Set
$$T_C := f_{m,p} \circ \dots \circ f_{1,2}  \circ f_{1,1}
(T),$$ 
so that $T_C$ is the matrix obtained from $T$ by the restoration algorithm. 

\begin{example} Let $m=p=3$ and  
consider the Cauchon diagram 
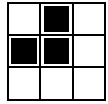
\[
\young(\wbox\bbox\wbox,\bbox\bbox\wbox,\wbox\wbox\wbox)
\]

Then $$T=\left( \begin{array}{ccc}
t_{1,1} & 0 & t_{1,3}\\
0 & 0 & t_{2,3} \\
t_{3,1} & t_{3,2} & t_{3,3} 
\end{array} \right)$$ 
and $T^{(j,\beta)}:= f_{j,\beta} \circ \dots \circ f_{1,1} (T)$.

Then we have   $$T^{(3,2)}=T^{(3,1)}=T^{(2,3)}=T^{(2,2)}=T^{(2,1)}=T^{(1,3)}=T^{(1,2)}=T$$ and 
$$T_C=T^{(3,3)}=\left( \begin{array}{ccc}
t_{1,1} +t_{1,3} t_{3,3}^{-1} t_{3,1}  & t_{1,3}t_{3,3}^{-1} t_{3,2}  & t_{1,3} \\
t_{2,3} t_{3,3}^{-1} t_{3,1} & t_{2,3} t_{3,3}^{-1} t_{3,2} & t_{2,3} \\
t_{3,1} & t_{3,2} & t_{3,3} 
\end{array} \right).$$ \fin\\
\end{example}

The above construction can be applied in a variety of situations. 
In particular, we have the following. \\

$\bullet$ If $K = \mathbb{R}$ and $T$ is nonnegative, then 
$T_C$ is TNN.   

$\bullet$ If the nonzero entries of $T$ commute and are algebraically independent, and if $K=\mathbb{C}(t_{ij})$, then the minors of $T_C$ that are equal to zero are exactly those that vanish on the closure of a given $\hc$-orbit of symplectic leaves. (See \cite{GLL}.)

$\bullet$ If the nonzero entries of $T$ are the generators of a certain quantum affine space over $\bbC$ and $K$ is the skew-field of fractions of this quantum affine space, then the quantum minors of $T_C$ that are equal to zero are exactly those belonging to the unique $\hc$-prime in $\mmpc$ associated to the Cauchon diagram $C$. (See \cite{Lau2} for more details.)

$\bullet$ The families of (quantum) minors we get depend only on $C$ in these
three cases. And if we start from the same Cauchon diagram in these three
cases, then we get exactly the same families. \\

As a consequence, we get the following comparison result (see \cite{GLL,GLL2}).

\begin{theorem}
\label{thm:main}
 Let $\mathcal{F}$ be a family of minors in the coordinate ring of $\mathcal{M}_{m,p}(\mathbb{C})$, and let $\mathcal{F}_q$ be the corresponding family of quantum minors in $\mathcal{O}_q(\mathcal{M}_{m,p}(\mathbb{C}))$. Then the following are equivalent:
\begin{enumerate}
\item The totally nonnegative cell associated to $\mathcal{F}$ is nonempty. 
\item $\mathcal{F}$ is the set of minors that vanish on the closure of a torus-orbit of symplectic leaves in $\mathcal{M}_{m,p}(\mathbb{C})$.
\item $\mathcal{F}_q$ is the set of quantum minors that belong to 
an $\ch$-prime in $\oqmmpc$.
\end{enumerate}
\end{theorem}

This result has several interesting consequences.

First, it easily follows from Theorem \ref{thm:main} that the TNN cells in
$\mathcal{M}_{m,p}^{\mathrm{tnn}}$ are the traces of the closure of
$\hc$-orbits of symplectic leaves on $\mathcal{M}_{m,p}^{\mathrm{tnn}}$.

Next, the sets of all minors that vanish on the closure of a torus-orbit of
symplectic leaves in $\mathcal{M}_{m,p}(\mathbb{C})$ have been explicitly described in \cite{GLL} (see also Theorem \ref{theorem-4.2reformulation}). So, as a consequence of the previous theorem, {\em the sets of
minors that define nonempty totally nonnegative cells are explicitly
described:} these are the families $\mathcal{M}(w)$ of Definition \ref{def:mw} for $w \in \mathcal{S}$. On the other hand, when the deformation parameter $q$ is transcendental over the rationals, then the torus-invariant primes in $\OMmpc$ are generated by quantum minors, and so we deduce from the above theorem {\em explicit generating sets of quantum minors for the torus-invariant prime ideals of $\mathcal{O}_q(\mathcal{M}_{m,p}(\mathbb{C}))$}. 
Recently and independently, Yakimov \cite{Yak} also described explicit families of quantum minors that generate $\hc$-primes. However his families are smaller than ours and so are not adapted to the TNN world. The problem of deciding whether a given quantum minor belongs to the $\hc$-prime associated to a Cauchon diagram $C$ has been studied recently by Casteels \cite{Ca} who gave a combinatorial criterion inspired by Lindstr\"om's Lemma.

\section*{Acknowledgements}

The results in this paper were announced during the mini-workshop
``Nonneg\-a\-tiv\-i\-ty is a quantum phenomenon'' that took place at the
Mathematisches For\-schungs\-in\-sti\-tut Oberwolfach, 1--7 March 2009,
\cite{MFO}; we thank the director and staff of the MFO for providing the ideal
environment for this stimulating meeting. We also thank Konni Rietsch, Laurent
Rigal, Lauren Williams, Milen Yakimov and, especially, our co-author Ken
Goodearl for discussions and comments concerning the results presented in this
survey paper both at the workshop and at other times. We also thank Natalia
Iyudu for the organisation of the Belfast conference in August 2009 
at which we presented many of these results.

\bibliographystyle{amsplain}
\bibliography{biblio}

\vskip 1cm

\providecommand{\bysame}{\leavevmode\hbox to3em{\hrulefill}\thinspace}
\providecommand{\href}[2]{#2}

\vskip 1cm
%\newpage 
\begin{minipage}{1.00\linewidth}
\noindent 
S Launois: \\
School of Mathematics, Statistics and Actuarial Science,\\
University of Kent\\
Canterbury, Kent CT2 7NF, UK\\
Email: {\tt S.Launois@kent.ac.uk} \\

\noindent 
T H Lenagan: \\
Maxwell Institute for Mathematical Sciences\\
School of Mathematics, University of Edinburgh,\\
James Clerk Maxwell Building, King's Buildings, Mayfield Road,\\
Edinburgh EH9 3JZ, Scotland, UK\\
E-mail: {\tt tom@maths.ed.ac.uk} 
\end{minipage}

 \end{document}